\pdfoutput=1

\documentclass[12pt]{article} \textheight 8in\textwidth
6in \oddsidemargin 0in\evensidemargin 0in
\usepackage{amsmath,amsthm,amssymb}
\usepackage{graphicx}
\usepackage{amsfonts}
\usepackage{epsf}
\usepackage{tikz}

\makeatletter
\let\@fnsymbol\@arabic
\makeatother

\newtheorem{theorem}{Theorem}[section]
\newtheorem{lemma}[theorem]{Lemma}
\newtheorem{corollary}[theorem]{Corollary}

\newtheorem{algorithm}[theorem]{Algorithm}

\newtheorem{definition}[theorem]{Definition}
\newtheorem{example}[theorem]{Example}

\def\pont{\hspace{-6pt}{\bf.\ }}
\newcommand\ceil[1]{\left\lceil#1\right\rceil}
\newcommand\floor[1]{\left\lfloor#1\right\rfloor}

\def\beq{\begin{equation}}\def\eeq{\end{equation}}
\def\beqn{\begin{eqnarray}}\def\eeqn{\end{eqnarray}}
\def\pont{\hspace{-6pt}{\bf\ }}

\def\qed{\ifhmode\unskip\nobreak\fi\quad\ifmmode\Box\else$\Box$\fi}

\begin{document}
\title{Remarks on the distribution of colors in Gallai colorings}
\author{Joseph Feffer$^{a}$\footnote{Corresponding author, Present address: Department of Mathematics, Harvard University, 1 Oxford St., Cambridge, MA 02138, United States;  jrfeffer@college.harvard.edu}, Yaoying Fu$^a$\footnote{Present address: Depament of Mathematics, Hong Kong University, Run Run Shaw Building, Pok Fu Lam Road, Hong Kong, China; leah1127@connect.hku.hk}, Jun Yan$^a$\footnote{Present address: Department of Mathematics, University of Waterloo, 200 University Ave W, Waterloo, ON N2L 3G1, Canada; j228yan@uwaterloo.ca}\\\\$^a$ Budapest Semesters in Mathematics REU\\ Alfr\'ed R\'enyi Institute of Mathematics, Hungarian Academy of Sciences\\ Bethlen G\'abor t\'er 2, Budapest, 1071 Hungary} 
\maketitle

\begin{abstract}
A Gallai coloring of a complete graph $K_n$ is an edge coloring without triangles colored with three different colors. A sequence $e_1\ge \dots \ge e_k$ of positive integers is an $(n,k)$-sequence if $\sum_{i=1}^k e_i=\binom{n}{2}$. An $(n,k)$-sequence is a G-sequence if there is a Gallai coloring of $K_n$ with $k$ colors such that there are $e_i$ edges of color $i$ for all $i,1\le i \le k$. Gy\'arf\'as, P\'alv\"olgyi, Patk\'os and Wales proved that for any integer $k\ge 3$ there exists an integer $g(k)$ such that every $(n,k)$-sequence is a G-sequence if and only if $n\ge g(k)$. They showed that $g(3)=5, g(4)=8$ and $2k-2\le g(k)\le 8k^2+1$.

We show that $g(5)=10$ and give almost matching lower and upper bounds for $g(k)$  by showing that with suitable constants $\alpha,\beta>0$, $\frac{\alpha k^{1.5}}{\ln k}\le g(k) \le \beta k^{1.5}$ for all sufficiently large $k$.

\end{abstract}

\textit{Keywords}: Gallai coloring, Gallai sequence, graph decomposition

\section{Introduction}
Gallai colorings (a term introduced in \cite{GYS} referring to a concept of Gallai \cite{GAL}) of complete graphs are edge colorings that do not contain triangles colored with three different colors. We shall abbreviate Gallai colorings as G-colorings. Observe that every $2$-coloring of the edges of the complete graph $K_n$ is a G-coloring but the number of colors in a G-coloring is not fixed.

Ramsey-type problems for G-colorings have been investigated in several papers such as \cite{GYS,GYSSS,LT,MS} and the statistical behaviour of the number of G-colorings have been studied in \cite{Betal,BBH,BL}. A recent work \cite{GyPPW} explored the distribution of colors in G-colorings and this note contributes to a problem exposed there.

We call a sequence $e_1,\dots ,e_k$ of nonnegative integers an $(n,k)$-sequence if $\sum_{i=1}^k e_i=\binom{n}{2}$. An $(n,k)$-sequence is a G-sequence if there is a G-coloring of $K_n$ with $k$ colors such that there are $e_i$ edges of color $i$ for all $i,1\le i \le k$.  A sequence $e_1,\dots ,e_k$ is {\em ordered} if $e_1\ge\dots\ge e_k$.

The following decomposition theorem  plays a central role in the theory of G-colorings. It was formulated in \cite{GYS} but implicitly it was already in \cite{GAL}.

\begin{theorem}\pont[Theorem A, \cite{GYS}]\label{decomp}
Assume that we have a G-coloring on $K_n$ with at least three colors. Then there exist at most two colors, say $1,2$, and a decomposition of $K_n$ into $m\ge 2$ vertex disjoint complete graphs $K_{n_i}$ ($1\le i \le m$) so that all edges between $V(K_{n_i})$ and $V(K_{n_j})$ are colored with the same color and that color is either $1$ or $2$.
\end{theorem}

We call the (at most) two colors in Theorem \ref{decomp} {\em base colors}, and call the edges within the disjoint complete graphs {\em internal edges}. Theorem \ref{decomp} can be stated in a slightly stronger form observing that if one of the base colors does not span a connected subgraph of $K_n$ then the other base color provides a decomposition of $K_n$ with one base color, spanning a connected subgraph of $K_n$. This leads to the following corollary.

\begin{corollary}\label{bigbase} There exists a decomposition according to Theorem \ref{decomp} where all (two or one) base colors span a connected subgraph of $K_n$. Thus in a suitable decomposition, all base colors have at least $n-1$ edges.
\end{corollary}

It was proved in \cite{GyPPW} that for any integer $k\ge 2$ there is a (unique) integer $g(k)$ with the following property: there exists a Gallai $k$-coloring of $K_n$ with $e_i$ edges in color $i$ for every $e_1, \dots , e_k$ satisfying $\sum_{i=1}^k e_i={n\choose 2}$, if and only if $n\ge g(k)$. The bounds $2k-2\le g(k)\le 8k^2+1$ were given in \cite{GyPPW}. It was also proven in \cite{GyPPW} that $g(3)=5$ and $g(4)=8$. In this paper, we improve the upper and lower bounds and show that $g(5)=10$. Specifically, we prove the following theorems:

\begin{theorem}\label{upperb} There exists a constant $\beta>0$ such that $g(k) \le \beta k^{3/2}$ for all sufficiently large $k$. In fact, $\beta\le \beta(k)$, where $\beta(k)\rightarrow 2\sqrt 3$ as $k\rightarrow\infty$.
\end{theorem}

\begin{theorem}\label{lowerb}
There exists a constant $\alpha>0$ such that $g(k)\ge \frac{\alpha k^{1.5}}{\ln k}$ for all sufficiently large $k$. In fact,  $\alpha \ge \alpha(k)$, where $\alpha(k)\rightarrow 1$ as $k\rightarrow\infty$. 
\end{theorem}

\begin{theorem}\label{g5}
 The $(9,5)$-sequence $12,6,6,6,6$ is not a G-sequence but all $(10,5)$-sequences are G-sequences. Thus $g(5)=10$.
\end{theorem}

In Section \ref{ubsect} we prove Theorem \ref{upperb} with an algorithm whose input is an $(n,k)$-sequence with an arbitrary $n$ satisfying $n\ge \beta(k) k^{3/2}$ and whose output is a $G$-coloring of $K_n$. The algorithm works using a series of simple decompositions called {\em cuts} that use one base color to separate the graph into $2$ parts according to Corollary \ref{bigbase}.  The algorithm ensures that the iterated decomposition process never stops with a part $K_a$ with $a>1$ because there is always a term in the actual sequence with value at least $a-1$ .

Section \ref{lb} is devoted to proving Theorem \ref{lowerb} regarding lower bounds on $g(k)$. We explicitly construct a family of color sequences and show that any possible decomposition steps for these sequences cannot decompose the complete graph down past a certain size. This implies that these sequences are not G-sequences and provide a lower bound on $g(k)$.

Section \ref{sectg5} gives the proof of Theorem \ref{g5}, reducing $(10,5)$-sequences by cuts to $4$-sequences realizable by G-colorings on a smaller complete graph. Although we found some tools to limit the number of cases within a plausible range, the full analysis is rather long, predicting that it is difficult to determine $g(k)$ exactly.

\newpage

\section{Proof of Theorem \ref{upperb}}\label{ubsect}

The proof of Theorem \ref{upperb} is separated into four subsections. The first defines an algorithm to provide G-colorings using simple decomposition steps. For technical reasons it is convenient to extend the definition of $(n,k)$-sequences allowing sequences $e_1,\dots,e_k$ with $\sum_{i=1}^k e_i \ge {n\choose 2}$ edges. The second subsection partitions $(n,k)$-sequences that are not G-colorable by the algorithm into classes, called {\em irreducibility classes}. In the third subsection we show that each class of $(2k,k)$-sequences has G-colorings if $\sum_{i=1}^k e_i$ is large enough. In the last subsection we prove Theorem \ref{upperb}.

\subsection{A $G$-Coloring Algorithm}

\begin{definition} Assume that $K_n$ is partitioned into two parts, a $K_{n-j}$ and a $K_j$.  The set of edges between them is called a {\bf cut}.
\end{definition}

\begin{algorithm}\label{alg}
For a given input $(n,k)$-sequence $s=(e_1,\ldots,e_k)$ (allowing $\sum_{i=1}^k e_i \ge {n\choose 2}$ as discussed above) we define an algorithm to G-color $K_n$ using only cuts to color all edges of the cut with the same color. A ``branch" of the algorithm works as follows:
\begin{enumerate}
    \item Define $S$ to be a set that will contain vertex disjoint complete graphs whose union contains all $n$ vertices. Initially let $S=\{K_n\}$.
    \item Pick $K_a\in S$ such that $a\ge b$ $\forall$ $K_b\in S$ and remove $K_a$ from $S$.
    \item Partition $K_a$ into two vertex disjoint complete subgraphs $K_{a_1}$ and $K_{a_2}$ so that $a_1a_2\le e_i$ for some $i$. If there are multiple colors that satisfy this inequality, choose one to be $i$. Color all edges of the cut $K_{a_1},K_{a_2}$ with color $i$. Adjust the actual sequence $e_1,\ldots,e_k$ by replacing $e_i$ with $e_i-a_1a_2$ and go to the next step (step 4).  If there is no cut with the condition, i.e. $e_i<a-1$ for all $i\in [k]$, exit this loop. In this case, he branch is terminated and called {\bf irreducible at $K_a$}.
    \item Add $K_{a_1}$ and $K_{a_2}$ to $S$.
    \item As long as $S$ contains a graph with more than one vertex, repeat from step $2$.
\end{enumerate}
\end{algorithm}

Each branch of algorithm \ref{alg} can be represented by a path on a graph where vertices correspond to iterations within the branch and are labeled with the actual partition $S$ and with the actual sequence $(e_1,...,e_k)$ at the beginning of the iterations. Furthermore, all of these branches can be linked together on a tree (see Example \ref{ex1}) with $S=\{K_n\}$ and the initial $(n,k)$-sequence at the root. The $G$-coloring algorithm is then just a depth-first search on this tree and will stop at the first place where $S=\{K_1,...,K_1\}$. If this never happens then all branches are irreducible. In this case the input sequence of the algorithm is called {\em irreducible} as well.\\
\begin{example}\label{ex1}
The following shows three branches of the decomposition of $K_{8}$ with the sequence $(14,8,3,3)$. Terms in $S$ that are $K_1$'s are not shown for brevity. Terms are not reordered after cuts for the sake of clarity:\\

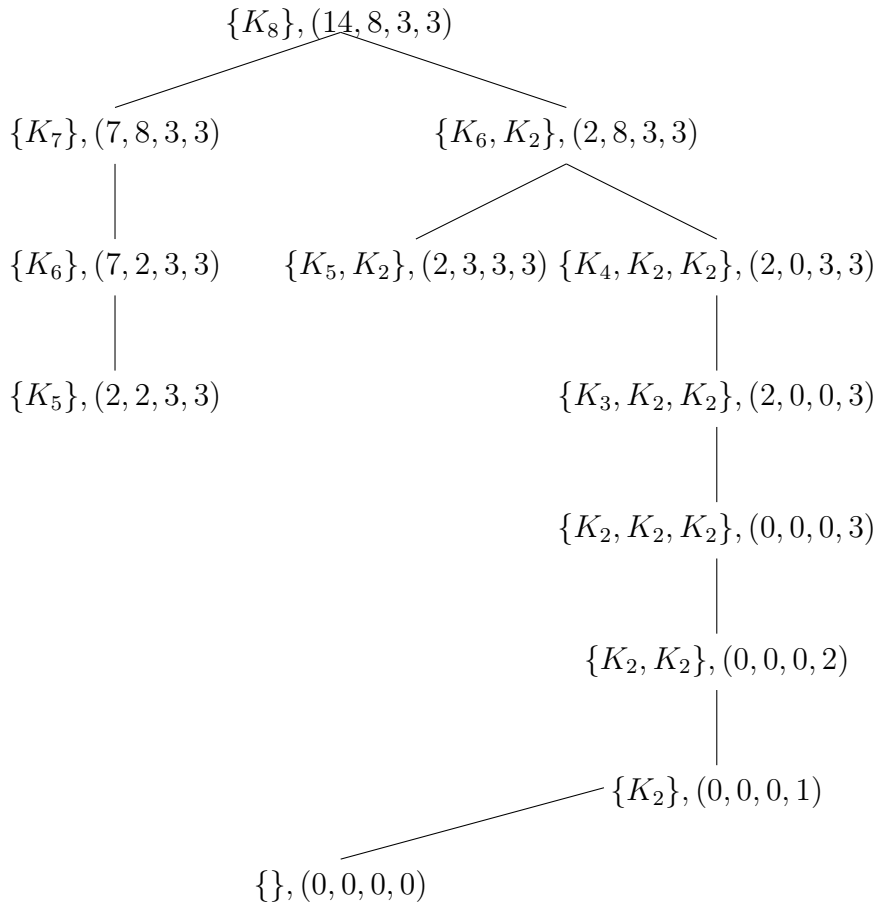
\hspace*{-1em}\begin{tikzpicture}
(0,.25) node{${\{K_{8}\},(14,8,3,3)}$};
\draw (-1.5,0) -- (-4.5,-1) node [at end, below] {${\{K_{7}\},(7,8,3,3)}$};
\draw (-4.5,-1.75) -- (-4.5,-2.75) node [at end, below] {$\{{K_{6}\},(7,2,3,3)}$};
\draw (-4.5,-3.5) -- (-4.5,-4.5) node [at end, below] {$\{{K_{5}\},(2,2,3,3)}$};
\draw (-1.5,0) -- (1.5,-1) node [at end, below] {${\{K_{6},K_2\},(2,8,3,3)}$};
\draw (1.5,-3.5+1.75) -- (-0.5,-4.5+1.75) node [at end, below] {${\{K_{5},K_2\},(2,3,3,3)}$};
\draw (1.5,-3.5+1.75) -- (3.5,-4.5+1.75) node [at end, below] {$\{K_{4},K_2,K_2\}, (2,0,3,3)$};
\draw (3.5,-5.25+1.75) -- (3.5,-6.25+1.75) node [at end, below] {$\{K_{3},K_2,K_2\}, (2,0,0,3)$};
\draw (3.5,-7+1.75) -- (3.5,-8+1.75) node [at end, below] {$\{K_{2},K_2,K_2\}, (0,0,0,3)$};
\draw (3.5,-8.75+1.75) -- (3.5,-9.75+1.75) node [at end, below] {$\{K_2,K_2\}, (0,0,0,2)$};
\draw (3.5,-8.75) -- (3.5,-9.75) node [at end, below] {$\{K_2\}, (0,0,0,1)$};
\draw (2,-9.75-.3) -- (-1.5,-11) node [at end, below] {$\{\}, (0,0,0,0)$};
\end{tikzpicture}
\end{example}

\subsection{Irreducibility Classes}

For an irreducible $(n,k)$-sequence each branch of algorithm \ref{alg} is irreducible and stops with a $K_p$ and a sequence $e_1,\dots,e_k$ with $e_i<p-1$ for all $i$. For different attempts to decompose $K_n$, branches may stop when trying to decompose different $K_p$'s. However, there are a finite number of distinct partial decompositions of $K_n$ that can be attempted, so we can define a minimum $p$ for each sequence. We call this $p$ the lowest stopping point of the decomposition of $K_n$ with this sequence.

\begin{definition}\label{iclass}
 For an $(n,k)$-sequence $s$ with sum $m$, let $f(s)$ be the lowest stopping point of the decomposition of $K_n$ with $s$. Then, for $p\ge 2$, we define the irreducibility class $I^n_p(m)$ as \[I^n_p(m)=\{s:f(s)=p\}\].
\end{definition}

\begin{lemma}\label{takeout}
For $n\ge 2kj-j+1$, a decomposition of $K_n$ with an $(n,k)$-sequence into two subgraphs $K_{n-j}$ and $K_{j}$ can always be performed. Additionally, for $j=1$, a decomposition can be performed for $n=2k-1$.
\end{lemma}
\begin{proof}
For $n = 2kj-j+1$,
\[n(n-1)= (2kj-j+1)(2kj-j) > 2kj(2kj-2j+1)\implies \frac{{n\choose 2}}{k}>j(n-j)\]
We also have that if the previous inequality is true for $n$ and $n\ge kj$,
\[{n+1\choose 2}={n\choose 2}+n>kj(n-j)+kj=kj(n+1-j)\implies \frac{{n+1\choose 2}}{k}>j(n+1-j)\]
Thus for all $n\ge 2kj-j+1$, we have that $\frac{{n\choose 2}}{k}>j(n-j)$. This means there must be some term in the $(n,k)$-sequence that is greater than $j(n-j)$ and can be used as the base color to remove a $K_j$.\\
Additionally, for $j=1$ and $n=2k-1,$
\[\left\lceil \frac{{2k-1\choose 2}}{k}\right\rceil = \left\lceil \frac{(2k-1)(2k-2)}{2k} \right\rceil = \left\lceil 2k-2 - \frac{2k-2}{2k}\right\rceil = 2k-2 \]
Thus there must be a color with at least $(2k-1)-1$ edges so it can be used as a base color to remove a $K_1$.
\end{proof}

\begin{corollary}\label{cor2k-1}
Fix a positive integer $k$. For $n$ and $p\ge 2k-1$ and for all $m\ge {n \choose 2}$, $I^n_p(m)=\emptyset$.
\end{corollary}
\begin{proof}
If $n<p$ then, because the decomposition algorithm \ref{alg} starts with a $K_n$ and decomposes it, there will never be a $K_p$ in $S$. For $n\ge p$, since $p\ge 2k-1$, a decomposition can always be performed on $K_p$  by Lemma \ref{takeout}.
\end{proof}

\begin{lemma}
 Assume that $k(p-2)\le m$ and $s=(e_1,\ldots,e_k)\in I^n_p(m)$ is an ordered $(n,k)$-sequence with $e_1\ge \dots \ge e_j>p-2\ge e_{j+1}\ge \dots \ge e_k$. Then, for some $q\ge p$, there exists an ordered $(n,k)$-sequence $s'=(f_1,\ldots,f_k)\in I^n_{q}(m)$ with $f_{j+1}=\dots =f_k=p-2$.
\end{lemma}
\begin{proof}
Let $t$ be the smallest integer in $[j+1,k]$ such that $e_t<p-2$. If no such $t$ exists then set $s'=s$. Otherwise, because $k(p-2)\le m$, we can define an ordered sequence $s_1$ as follows. Increase $e_t$ by one and decrease some $e_i$ with $i\in [1,j]$ by one.  We can repeat these modifications to get $s_2,\dots,s_w$  such that all terms of $s_w$ with index at least $j+1$ are equal to $p-2$ and all terms with index at most $j$ are at least $p-2$. Set $s'=s_w$. Then algorithm \ref{alg} on $s'$ cannot use colors other than the first $j$ to decompose any $K_a$. Because $f_i\le e_i$ for $i\in [1,j]$ and the first $j$ colors of $s$ could not decompose $K_n$ past $K_p$, the first $j$ colors of $s'$ cannot decompose $K_n$ past $K_p$.
\end{proof}

\begin{corollary}\label{tails}
If there are no ordered $(n,k)$-sequences ending in a nonnegative number of $(p-2)$'s in any $I^n_q(m)$ for $q\ge p$, then $I^n_p(m)=\emptyset$.
\end{corollary}

When we say a nonnegative number of $(p-2)$'s, and this number is $0$, what we really mean is a sequence in which all of the terms are greater than $(p-2)$. However, phrasing it in this way will be useful in the next section.

\subsection{Tail Length Classes}

\begin{definition}
For an ordered $(n,k)$-sequence, $s=(e_1,\ldots,e_k)$ and a fixed value $r$,  define the tail length $l_r(s)$ to be the number of terms in the sequence less than or equal to $r$. For $0 \le l\le k$, Let the tail length class $L_{r}(l)$ be defined as
\[L_{r}(l)=\{s:l_r(s)=l\}\]
\end{definition}

Note that for any fixed $r$, the sets $\{L_r(l)\}$ partition the set of $(n,k)$-sequences.

\begin{lemma}\label{yeet}
Fix the integers $k>0$ and $p<2k-1$. If there is an $m$ such that $I_q^{2k}(m)=\emptyset$ for all $q\ge p+1$, then $I_{q'}^{2k}(m+x)=\emptyset$ for all $q'\ge p$ where $x=min(k-1,p-1)$.
\end{lemma}
\begin{proof} Suppose $I_{q'}^{2k}(m+x)\ne \emptyset$ for some $q'\ge p$ where $x=min(k-1,p-1)$. Then there is an ordered sequence $s=(e_1,\ldots.,e_k)$ with sum at least $m+x$ decomposable to $K_{q'}$ but not any further.

\noindent
{\bf Case 1. $q'>p$}.

Because $x<2k-1$, $s'=(e_1-x,\ldots,e_k)$ is a sequence with $k$ terms and with sum at least $m$ and can thus be decomposed past $K_{q'}$. This contradicts the assumption of the lemma.

\noindent
{\bf Case 2. $q'=p$}. We know from Case 1 that $I^{2k}_{q'}(m+x)=\emptyset$ for all $q'>p$ so we can always decompose $(e_1,\ldots,e_k)$ at least until $K_{p}$. Applying Corollary \ref{tails} with $n=2k$ and $m+x$ as the sequences' sum, we have that $s\in L_{p-2}(l)$ for some $l$. We separate this case into two subcases.

\noindent
{\bf Subcase 2.1.} $s\in L_{p-2}(l)$ with $l\le p-k-1$.

 Each base color in a decomposition of $K_n$ removes at least a $K_1$ from it. Thus because $2k-p$ vertices are removed in the series of decompositions from $K_{2k}$ to $K_{p}$, a maximum of $2k-p$ colors could have been used as base colors to decompose to $K_{p}$. The total number of $e_i>p-2$ is
\[k-l\ge k-(p-k-1)=2k-p+1.\]
Thus there must be at least one $e_i$ with $e_i> p-2$ that has not been used as a base color. We have that this $e_i\ge p-1$ so we can use it to decompose the $K_{p}$, contradicting the assumption of the lemma.\\

\noindent
{\bf Subcase 2.2.} $s\in L_{p-2}(l)$ with $l\ge p-k$.

Write $s=(e_1,\ldots,e_j,p-2,\ldots,p-2)$ (where $j=k$ means no trailing $p-2$'s). Since the sum of $s$ is $m+ \min(k-1,p-1)$ and $p<2k-1$, we have $e_1>p-1$.  Consider the sequence $s'=(e_1-(p-1),e_2,..,e_j,p-1,\ldots, p-1)$  with $l$ $p-1$'s. The sum of $s'$ is
\[(m+x)-(p-1)+l.\]
If $p-k\ge 0$, then $k-1\le p-1$ and
\[(m+x)-(p-1)+l\ge (m+k-1)-(p-1)+(p-k)=m.\]
If $p-k<0$, then $p-1<k-1$ and because $l\ge 0,$
\[(m+x)-(p-1)+l\ge (m+p-1)-(p-1)+0\ge m.\]
Thus the sequence $s'$ will always have sum greater than or equal to $m$. By the given, it can be reduced down to $K_{p}$ and by Corollary \ref{bigbase} this can be done using only the first $j$ colors. Then for $s$  the same decomposition can be used to guarantee an $e_i\ge p-1$. This can then be used to decompose the $K_{p}$.
Thus no $(2k,k)$-sequence ending in a nonnegative number of $p-2$'s is in $I_{p}^{2k}(m+x)$ so by Corollary \ref{tails}, $I_{p}^{2k}(m+x)=\emptyset$.
\end{proof}

\begin{corollary}\label{moreedges}
If $s=(e_1,\ldots,e_k)$ is a $(2k,k)$-sequence with $e_1+\ldots+e_k= {2k\choose 2} + \frac{3k^2-7k+2}{2}$, then $s$ is $G$-colorable.
\end{corollary}
\begin{proof}
\begin{itemize}
\item $p\ge 2k-1$. Corollary \ref{cor2k-1} with $n=2k,m={2k\choose 2}$ implies $I_{p}^{2k}({2k\choose 2})=\emptyset$.
\item $p\in [k,2k-2]$.  Lemma \ref{yeet} implies that \[\forall\: q\ge p+1, I_{q}^{2k}(m)=\emptyset \implies \forall\: q\ge p, I_{q}^{2k}(m+k-1)=\emptyset.\] Performing this for all values of $p$ in this range we get
\[m={2k\choose 2}+(k-1)^2 \implies \forall\: q\ge k, I_q^{2k}(m)=\emptyset.\]
\item $p\in [3,k-1]$.  Lemma \ref{yeet} gives that
\[\forall\:q\ge p+1, I_{q}^{2k}(m)=\emptyset  \implies  \forall\: q\ge p+1, I_q^{2k}(m+p-1)=\emptyset.\]
Performing this for all values of $p$ in this range gives that
\[m={2k\choose 2} + (k-1)^2+{k-1\choose 2}-1={2k\choose 2}+\frac{3k^2-7k+2}{2}\implies \forall\: q\ge 3, I_q^{2k}(m)=\emptyset.\]
\end{itemize}

We have that $I_2^{2k}(m)$ is empty because if there is an edge left, it can always be used to color a $K_2$.    Additionally, observing that the decomposition of Algorithm \ref{alg} ensures that nothing bigger than a $K_1$ is added to $S$ during the decomposition of $K_{2k}$, there are no extra components that may be too large to color.
\end{proof}

\subsection{G-coloring of $K_n$}

To prove Theorem \ref{upperb} we will use Lemma \ref{takeout} to decompose $K_n$ (for $n$ large enough) down to $K_{2k}$ in such a way that $\sum_{K_i\in S}{i\choose 2}\ge \frac{3k^2-7k+2}{2}$. By Corollary \ref{moreedges}, we can then color the $K_{2k}$. Letting $n=2k(n_0+1)$, for some $n_0$, we will remove as many $K_{n_0}$'s as possible from the $K_a$'s with $a\in [2kn_0,2k(n_0+1)]$ which we can do by Lemma \ref{takeout}. We will then remove as many $K_{n_0-1}$'s as possible from the $K_a$'s with $a\in [2k(n_0-1),2kn_0]$. Each time we remove a $K_j$ from $K_a$, the next largest graph in $S$ will be $K_{a-j}$. To find a lower bound on the number of $K_j$ we can remove from the $K_a$ with $a\in [2kj,2k(j+1)]$, we should assume we have removed a $K_{j+1}$ from $K_{2k(j+1)}$ and remove as many $K_j$'s as we can from the $K_a$ with $a\in [2kj,2k(j+1)-j-1]$, removing one from $K_{2kj},K_{2kj+j},...,K_{2kj+yj}$ where $yj\le 2k-j-1$. Solving this yields
\[y=\left\lfloor \frac{2k-j-1}{j}\right\rfloor.\]
So, the number of $K_j$'s that we will remove is $\left \lfloor \frac{2k-1}{j}\right\rfloor$.
Thus we want our $n_0$ to satisfy
\[\sum_{i=2}^{n_0}\left\lfloor\frac{2k-1}{i}\right\rfloor{i\choose 2}>\frac{3k^2-7k+2}{2}\]
All that is left to show is that we can color the remaining small components in $S$ after $K_{2k}$ is decomposed and that we can successfully determine the value of $n_0$. To do the first part, we will show that when trying to decompose a $K_j\in S$, there is always a color with greater than or equal to $j-1$ edges. For $j\ge 3$, when attempting to color $K_j$, we will not have colored any of the $K_{j-1}$. Thus the number of edges we will have left to color $K_j$ will be greater than
\[\left\lfloor\frac{2k-1}{j-1}\right\rfloor{j-1\choose 2}+{j\choose 2}\ge \left(\frac{2k-(j-1)}{j-1}\right){j-1\choose 2}+{j\choose 2}\]\[=\frac{(2k-(j-1))(j-2)+j(j-1)}{2} = k(j-2)+j-1>k(j-2)\]
Thus a color will always have enough edges to decompose $K_j$ for every $K_j\in S$ with $j\ge 3$. For $j=2$, if there are still edges, $K_j$ can be decomposed. Thus every element in $S$ is decomposable and our coloring is complete.

To calculate $n_0$, we have that
\[\sum_{i=2}^{n_0}\left\lfloor\frac{2k-1}{i}\right\rfloor{i\choose 2}\ge \sum_{i=2}^{n_0}\frac{(2k-i)(i-1)}{2}\]
If the right side is bigger than $\frac{3k^2-7k+2}{2}$, then $n_0$ will be as large as we need it to be. Thus we would like the smallest $n_0$ that makes the expression below positive.
\[2\left(\sum_{i=2}^{n_0}\frac{(2k-i)(i-1)}{2}-\frac{3k^2-7k+2}{2}\right)=\sum_{i=2}^{n_0}(-i^2+(2k+1)i-2k) -(3k^2-7k+2)\]
\[=-\left(\frac{n_0(n_0+1)(2n_0+1)}{6}-1\right)+(2k+1)\left(\frac{n_0(n_0+1)}{2}-1\right)-2k(n_0-1)-3k^2+7k-2\]
\begin{equation}\label{eq2.4}
=-\frac{n_0^3}{3}+kn_0^2+\left(\frac{1}{3}-k\right)n_0-3k^2+7k-2
\end{equation}
We would like (\ref{eq2.4}) to be positive.  For $n_0=2\sqrt{k}$ (\ref{eq2.4}) becomes
\[k^2-\frac{14}{3}{k\sqrt k}+7k+\frac{2\sqrt k}{3} -2.\]
By AM-GM,
\[k^2+6k\ge 2k\sqrt{6k}> \frac{14}{3}k\sqrt k.\]
Thus for $k\geq 2$,
\[k^2-\frac{14}{3}{k\sqrt k}+7k+\frac{2\sqrt k}{3} -2\ge k+ \frac{2\sqrt k}{3} -2> 0.\]

Thus for $n_0=2\sqrt{k}$, (\ref{eq2.4}) is positive and for $n_0>\sqrt{3k}$, the coefficient of the highest degree term in (\ref{eq2.4}), $kn_0^2-3k^2$, is positive.
We have that as $k$ gets large, $n_0$ approaching $\sqrt{3k}$ (thus $n=2k(n_0+1)$ approaching $2\sqrt{3}k^{3/2}$) will make (\ref{eq2.4}) positive, proving Theorem \ref{upperb}.

\newpage

\section{Proof of the Lower Bound}\label{lb}

We first prove an important lemma used throughout this section. It provides a lower bound on the size of the largest component in a decomposition step.
\begin{lemma}\label{boundcomp}
If the largest component in a partition of $K_n$ into disjoint complete graphs has size at most $j$ with $n>j>\frac{n}{2}$, then there are at most $\binom{j}{2}+\binom{n-j}{2}$ internal edges within the components of this partition.
\end{lemma}
\begin{proof}
Let the $K_{x_1}\cup\cdots\cup K_{x_m}$ be a partition of $K_n$ with the largest amount of internal edges. Without loss of generality assume $x_1\geq\cdots\geq x_m$. If $x_1<j$, then $m>1$ and $K_{x_1+1}\cup\cdots\cup K_{x_m-1}$, where we discard the last part if $x_m=1$, is also a partition of $K_n$. We have that
$$\binom{x_1+1}{2}+\binom{x_m-1}{2}-\binom{x_1}{2}-\binom{x_m}{2}=x_1-x_m+1>0,$$
so the new partition has more internal edges than the original one, which is a contradiction. Therefore $x_1=j$.

Similar to above, if $x_2<n-j$, then $m>2$, and $K_{x_1}\cup K_{x_2+1}\cdots\cup K_{x_m-1}$, where we discard the last part if $x_m=1$, is a partition of $K_n$ with more internal edges, contradiction. So $x_2=n-j, m=2$ and the maximum number of internal edges is indeed $\binom{j}{2}+\binom{n-j}{2}$.
\end{proof}

We now prove Theorem \ref{lowerb} by constructing a family of color sequences that are not G-sequences. Every color sequence we construct will have roughly half of the entries being large and the other half being small. By Corollary \ref{bigbase}, if this sequence is a G-sequence, then we can decompose the complete graph down to a certain stage without using colors corresponding to the small entries as base colors.

On the other hand, we show that the number of edges from colors corresponding to large entries are actually not enough for us to get to that stage. At each intermediate stage we have a lower bound for the number of internal edges in any possible decomposition. Lemma \ref{boundcomp} then gives a lower bound on the size of the largest component in any possible decomposition. We then use this information to count the minimum number of edges needed to decompose down to that desired stage and reach a contradiction.

\begin{proof}[\bf{Proof of Theorem \ref{lowerb}}]The proof consists of four steps.\\

\noindent{\bf Step 1: Constructing the sequence}

For every fixed integer $k\geq2$, let $f(k)=\floor{\frac{\alpha k^{1.5}}{\ln k}}$, where $\alpha$ is a constant to be chosen later. For simplicity, we shall write $f$ instead of $f(k)$ when there is no confusion. We want to construct an $(f,k)$-sequence $(e_1,\cdots,e_k)$ of the form $e_1=\cdots=e_{c}=a+1$, $e_{c+1}=\cdots=e_{\ceil{\frac{k}{2}}}=a$ and $e_{\ceil{\frac{k}{2}}+1}=\cdots=e_{k}=b$, such that $a>b$ and both are carefully chosen to make the ensuing calculation work. A suitable choice is $$b=3\ceil{\frac{f}{\sqrt{k}}}, a=\floor{\frac{\binom{f}{2}-\floor{\frac{k}{2}}b}{\ceil{\frac{k}{2}}}}, c=\binom{f}{2}-b\floor{\frac{k}{2}}-a\ceil{\frac{k}{2}}.$$ 
One can verify that $c<\ceil{\frac{k}{2}}$ and $c(a+1)+(\ceil{\frac{k}{2}}-c)a+\floor{\frac{k}{2}}b=c+a\ceil{\frac{k}{2}}+b\floor{\frac{k}{2}}=\binom{f}{2}$ so this is a well-defined $(f,k)$-sequence. One estimate that will be useful later is that $$a\leq\frac{\binom{f}{2}-\floor{\frac{k}{2}}b}{\ceil{\frac{k}{2}}}\leq\frac{f^2-f-(k-1)b}{k}<\frac{f^2}{k}=\frac{1}{9}(\frac{3f}{\sqrt{k}})^2\leq\frac{b^2}{9}.$$

We will show that this is not a G-sequence for suitable choices of $\alpha$. If this is a G-sequence, then by repeated applications of Theorem \ref{decomp} we can decompose $\{K_f\}$ down to $\{K_1,\cdots,K_1\}$. Moreover, by Corollary \ref{bigbase}, when decomposing $K_x$ with $x\geq b+2$, we can arrange so that the base colors have have at least $x-1\geq b+1$ edges. This implies that we can decompose $K_f$ to the stage where all components have size less than $b+2$ only using the first $\ceil{\frac{k}{2}}$ colors as base colors.\\

\noindent{\bf Step 2: Bounding the size of the largest component}

To bound the size of the largest component in each intermediate decomposition step, we will prove the following inequality for sufficiently large $k$ and all $b+2\leq x\leq f$:
\begin{equation}\label{intedgecond}
\binom{x}{2}-2(a+1)>\binom{x-h(x)}{2}+\binom{h(x)}{2},
\end{equation}
where $h(x)=\ceil{\frac{3(a+1)}{x}}$. 

Indeed, (\ref{intedgecond}) is equivalent to:
\begin{equation}\label{hcond}
x\cdot h(x)>2(a+1)+h(x)^2.
\end{equation}

Since $\frac{2}{3}x\cdot h(x)\geq 2(a+1)$, to show that (\ref{hcond}) holds, it suffices to show that $\frac{1}{3}x\cdot h(x)>h(x)^2$, i.e. $x>3\cdot h(x)$ for all $b+1<x\leq f$. Because $h(x)$ decreases as $x$ increases, it suffices to show this for $x=b+2$. We have that for all sufficiently large $k$ and thus sufficiently large $b$:
$$(b+2)^2>b^2+3b+24>9a+3b+24=9(a+2)+3(b+2)$$
$$\implies b+2>3(\frac{3(a+2)}{b+2}+1)>3\ceil{\frac{3(a+2)}{b+2}}=3\cdot h(b+2).$$

Therefore (\ref{hcond}) holds and so (\ref{intedgecond}) is proved.

Because there are at most two base colors, each with at most $a+1$ edges, the left side of (\ref{intedgecond}) is a lower bound on the number of internal edges for any decomposition of $K_x$. By Lemma \ref{boundcomp} and the fact that $x>3\cdot h(x)\implies x-h(x)>\frac{x}{2}$, the right side of (\ref{intedgecond}) is an upper bound on the number of internal edges for decompositions of $K_x$ if the largest component have size at most $x-h(x)$. Thus inequality (\ref{intedgecond}) implies that the largest component in any decomposition of $K_x$ must have size at least $x-h(x)+1$.\\

\noindent{\bf Step 3: Counting number of edges needed from the first $\ceil{\frac{k}{2}}$ colors}

We now count the minimum number of edges needed from the first $\ceil{\frac{k}{2}}$ colors to get to the desired stage where all components have size less than $b+2$.

For any such decomposition steps down to the desired stage, we can find the following intermediate decomposition steps $K_{x_0}\rightarrow K_{x_1}\rightarrow\cdots\rightarrow K_{x_j}$, where $f=x_0>x_1>\cdots>x_{j-1}\ge b+2>x_j$ and $K_{x_{r+1}}$ is the largest component in the decomposition of $K_{x_{r}}$ for all $0\leq r\leq j-1$. For all $0\le r\le j-1$, we have that $x_{r+1}\geq x_r-h(x_r)+1>x_r-h(x_r)$ by Step 2 and for any vertex in $K_{x_r}$ but not in the component $K_{x_{r+1}}$, all of its $x_{r+1}$ edges to $K_{x_{r+1}}$ are colored with base colors. Let $t_r=x_r-x_{r+1}$, then at least $t_rx_{r+1}>t_r(x_r-h(x_r))$ edges from the first $\ceil{\frac{k}{2}}$ colors are needed to decompose $K_{x_r}$. Therefore at least
\begin{align*}
\sum_{r=0}^{j-1}t_r(x_r-h(x_r))&=\sum_{r=0}^{j-1}\sum_{i=0}^{t_r-1}(x_r-h(x_r))\\
&>\sum_{r=0}^{j-1}\sum_{i=0}^{t_r-1}(x_r-i-h(x_r-i))\\
&=\sum_{x=x_j+1}^{x_0}(x-h(x))\geq \sum_{x=b+2}^{f}(x-h(x))
\end{align*}
edges from the first $\ceil{\frac{k}{2}}$ colors are needed to decompose $K_f$ into components of size less than $b+2$.

\begingroup
\allowdisplaybreaks
Further calculation shows that at least
\begin{align*}
\sum_{x=b+2}^{f}(x-h(x))&=\sum_{x=b+2}^{f}x-\sum_{x=b+2}^{f}\ceil{\frac{3(a+1)}{x}}\\
&>\frac{1}{2}(f+b+2)(f-b-1)-3(a+1)\sum_{x=b+2}^{f}\frac{1}{x}-\sum_{x=b+2}^{f}1\\
&>\frac{f^2-f-b^2-b}{2}-3(a+1)\ln\frac{f}{b}\\
&\geq\frac{f^2-f-b^2-b}{2}-3(a+1)\ln\frac{f\sqrt{k}}{3f}\\
&=\frac{f^2-f-b^2-b}{2}-3(a+1)\ln\frac{\sqrt{k}}{3}\\
&>\frac{f^2-f-b^2-b}{2}-\frac{3}{2}(a+1)\ln k
\end{align*} edges from the first $\ceil{\frac{k}{2}}$ colors are needed.
\endgroup

However, there are only
$$\binom{f}{2}-\floor{\frac{k}{2}}b\leq\frac{f^2-f}{2}-\frac{(k-1)b}{2}$$
edges in total from the first $\ceil{\frac{k}{2}}$ colors. \\

\noindent{\bf Step 4: Choosing constant $\alpha$ to reach a contradiction}

So if we have
$$\frac{f^2-f-b^2-b}{2}-\frac{3}{2}(a+1)\ln k>\frac{f^2-f}{2}-\frac{(k-1)b}{2}$$
\begin{equation}\label{contracond}
\iff(k-2)b>3(a+1)\ln k+b^2,
\end{equation}
then there is a contradiction because there are not enough edges from the first $\ceil{\frac{k}{2}}$ colors to decompose down to the desired stage. 

Recall that $a<\frac{b^2}{9}$. Fix a small $0<\epsilon<1$ so that the right side of (\ref{contracond}) is less than
$$(3+\epsilon)a\ln k+b^2<b^2(\frac{3+\epsilon}{9}\ln k+1)<(3+\frac{3f}{\sqrt{k}})^2\frac{(3+2\epsilon)\ln k}{9}$$
$$<\frac{(3+2\epsilon)(1+\epsilon)f^2\ln k}{k}<\frac{(3+7\epsilon)f^2\ln k}{k}$$
for all sufficiently large $k$. The left side of (\ref{contracond}) is at least $$(k-2)\frac{3f}{\sqrt{k}}>(3-\epsilon)\sqrt{k}f$$ for all sufficiently large $k$.
If $\alpha<\frac{3-\epsilon}{3+7\epsilon}$, then $$f(k)=\floor{\frac{\alpha k^{1.5}}{\ln k}}\implies (3-\epsilon)\sqrt{k}f\geq\frac{(3+7\epsilon)f^2\ln k}{k},$$ so (\ref{contracond}) holds for all sufficiently large $k$ and the original sequence is not a G-sequence. So we have $g(k)>f(k)\implies g(k)>\frac{\alpha k^{1.5}}{\ln k}$ for all sufficiently large $k$. Note that $\lim_{\epsilon\rightarrow0}\frac{3-\epsilon}{3+7\epsilon}=1$, so $\alpha(k)\rightarrow1$ as $k\rightarrow\infty$. This completes the proof.
\end{proof}

\newpage

\section{Proof of Theorem \ref{g5}}\label{sectg5}

The proof of Theorem \ref{g5} is based on some lemmas. 
\begin{lemma}\label{spec}There are four $(6,4)$-sequences that are not G-sequences: $(7,4,2,2),(7,3,3,2)$, $(6,3,3,3)$ and $(4,4,4,3)$. The only $(7,4)$-sequence that is not a G-sequence is $(9,4,4,4)$.
\end{lemma}

\begin{proof}
To prove the first part, let $s=(e_1,e_2,e_3,e_4)$ be a $(6,4)$ sequence.
\begin{itemize}
\item $e_1\le 4$.  We have the only sequence is $s=(4,4,4,3)$. This is not a G-sequence as it violates Corollary \ref{bigbase}.
\item $e_i=5$ for some $i$.  Then with the star $K_{1,5}$ we can reduce the problem to a $(5,3)$-sequence and it is a G-sequence since $g(3)=5$.
\item $e_1=6$. With the star $K_{1,5}$ we can reduce the problem to the $(5,4)$-sequence obtained by adding $e'=1$ to the sequence $(e_2,e_3,e_4)$. Three of the four possibilities are G-sequences (one uses a $K_3\cup K_2$ decomposition), the exception is $(6,3,3,3)$.
\item $e_1=7$. With the star $K_{1,5}$ we can reduce the problem to the $(5,4)$-sequence. Three of the five possibilities are G-sequences (two use $K_1\cup K_4$ decomposition), the exceptions are $(7,3,3,2)$ and $(7,4,2,2)$.
\item $e_1=8$. $K_2\cup K_4$ can be used.
\item $e_1=9$. $K_3\cup K_3$ can be used.
\item $e_1=10$. $K_3\cup K_3$ can be used except that $(10,3,1,1)$ uses $K_4\cup K_1\cup K_1$ and gets reduced to  $(2,2,1,1)$ on $K_4$.
\item $e_1=11$. $K_3\cup K_3$ can be used.
\item $e_1=12$. $K_2\cup K_2\cup K_2$ can be used.
\end{itemize}

To prove the second part, let $s=(e_1,e_2,e_3,e_4)$ be a $(7,4)$-sequence.
\begin{itemize}
\item $e_1=6$. With the star $K_{1,6}$ we can reduce $(7,4)$-sequences into $(6,3)$-sequences, which are always G-sequences since $g(3)=5$.
\item $e_1=7, 8,$ or $9$. With the star $K_{1,6}$ we can reduce $(7,4)$-sequences into $(6,4)$-sequences. We know that there are four $(6,4)$-sequences that are not G-sequences, and in this case there are in total five $(7,4)$-sequences which can become one of those non G-sequences on $K_6$. Four of the five possibilities are G-sequences (one uses $K_4\cup K_2\cup K_1$ decomposition), the exception is $(9,4,4,4)$.
\item $e_1=10,11$. $K_2\cup K_5$ can be used.
\item $e_1\ge 12$. $K_3\cup K_4$ can be used.\qedhere
\end{itemize}
\end{proof}

\begin{lemma}\label{big}
For any positive $j$ and let $k=2j-1,n=2g(j)$, any $(n,k)$-sequence with $e_1\ge g(j)^2$ is a G-sequence.
\end{lemma}
\begin{proof}
The proof of this statement relies on a decomposition into two copies of $K_{g(j)}$. Let $a_1\ge\ldots\ge a_{2j-1}$ be the number of edges left in each color after $g(j)^2$ edges of color $e_1$ are used to color the edges between the two copies of $K_{g(j)}$. We have that
\[a_1+a_2+\ldots.+a_{2j-1}={2g(j) \choose 2} - g(j)^2\]
We have that $a_{2i-1}\ge a_{2i}$ so
\[a_2+\ldots.+a_{2j-2}\le a_1+a_3+\ldots+a_{2j-3}\implies a_2+\ldots+a_{2j-2}\le\frac{{2g(j) \choose 2} - g(j)^2}{2}\]
Additionally, we have that $a_{2i}\ge a_{2i+1}$ so
\[a_3+\ldots.+a_{2j-1}\le a_2+a_4+\ldots+a_{2j-2}\implies a_3+\ldots.++a_{2j-1}\le\frac{{2g(k) \choose 2} - g(j)^2}{2}\]
Thus, we can place all edges of color $a_2,a_4,\ldots,a_{2j-2}$ in one copy of $K_{g(j)}$ and all edges of colors $a_3,a_5,\ldots,a_{2j-1}$ in the other, with edges of color $a_1$ filling remaining edges in each graph. Because we will have at most $j$ colors in each copy of $K_{g(j)}$, we will be able to Gallai color each one.
\end{proof}

To prove $g(5)=10$, in the Lemma \ref{big} we take $j=3$, then $k=2j-1=5, n=2g(j)=10$ and $g(3)=5$. It follows that if $e_1 \geq g(3)^2=25$, then the $(10,5)$-sequence is always a G-sequence. Thus we may assume that $e_1\le 24$.

\begin{lemma}\label{terms}
Let $s=(e_1,e_2,e_3,e_4,e_5)$ be a $(10,5)$-sequence. If $s$ contains a term in $\{7,8,9,14,15,16,17,21,22,23,24\}$, then $s$ is a G-sequence.
\end{lemma}
\begin{proof}
Let $s=(e_1,e_2,e_3,e_4,e_5)$ be a $(10,5)$-sequence.
\begin{itemize}
\item If $s$ contains 9, then we remove a $K_{1,9}$. What remains is a $(9,4)$-sequence. Since $g(4)=8$, we conclude that $s$ is a G-sequence.
\item If $s$ contains 8, then $s$ must contain a term $\geq 9$. Firstly we remove a $K_{1,9}$ from that term and then remove a $K_{1,8}$ from the term 8 in $s$. What remains is a $(8,4)$-sequence which is always a G-sequence since $g(4)=8$.
\item If $s$ contains 7, then $s$ contains either one term $\geq 17$ or one term $\geq 9$ and another $\geq 8$, and thus we could always remove a $K_{1,9}$ and subsequently a $K_{1,8}$. Afterwards we remove a $K_{1,7}$ from the term 7. What remains is a $(7,4)$-sequence. By Lemma \ref{spec}, we're done if we can prove that those $(10,5)$-sequences which become sequences in Lemma \ref{spec} after such coloring, called suspicious sequences in the rest of our proof, are all G-sequences on $K_{10}$. Suspicious sequences are:

$(26,7,4,4,4) \rightarrow$ $e_1=26>25$, it's a G-sequence by Lemma \ref{big}.

$(21,9,7,4,4) \rightarrow$ it's a G-sequence because it contains 9.

$(18,12,7,4,4) \rightarrow (18,12,7,4,4) - (8,9,7,0,0) =(10,4,4,3)$ on $K_7$

$(17,13,7,4,4)\rightarrow (17,13,7,4,4)-(9,8,7,0,0)=(8,5,4,4)$ on $K_7$.
\item If $s$ contains a term in $\{14,15,16,17,21,22,23,24\}$, then we use the decomposition $K_{10}\rightarrow K_7 \cup K_2 \cup K_1$
\item For 14, we use it up to color the edges between $K_7$ and $K_2$. We may assume there is no 9,8, or 7 in $s$ as otherwise it will be a G-sequence as aforementioned. Thus, other than 14, $s$ contains a term $\geq 10$. Then we use $1+2+7=10$ from that term to color the edge within $K_2$, edges between $K_1$ and $K_2$, and those connecting $K_1$ and $K_7$. Suspicious sequences are:

$(19,14,4,4,4)\rightarrow(19,14,4,4,4)-(9,14,1,0,0)=(10,4,4,3)$ on $K_7$

$(14,14,9,4,4)\rightarrow$ G-sequence because it contains a 9.
\item For 15, we use it up to color the edges between $K_7$ and $K_2$ and the one within $K_2$. We may assume there's no 9,8, or 7 in $s$, and thus we can use $9$ from another term in $s$ for a $K_{1,9}$. Suspicious sequences are:

$(18,15,4,4,4)\rightarrow$ $K_6 \cup K_3 \cup K_1\rightarrow$ Use 18 up for edges between $K_6$ and $K_3$, 9 from 15 for a $K_{1,9}$, and 3 from 4 for edges within $K_3\rightarrow(6,4,4,1)$ on $K_6$.

$(15,14,9,4,4)\rightarrow$ G-sequence because it contains 9.

\item For 16, we use it up to color the edges between $K_7$ and $K_2$, and the edges connecting $K_2$ and $K_1$. We may assume there's no 9,8, or 7 in $s$, and thus we can use $1+7=8$ from another term on the edge within $K_2$ and those between $K_7$ and $K_1$. Suspicious sequences are:

$(17,16,4,4,4)\rightarrow K_4\cup K_4\cup K_2\rightarrow$ Use 16 up to color all edges between $K_4$ and $K_4$, and 16 out of 17 to color edges between $K_4$ and $K_2$ and one on the edge within $K_2 \rightarrow (4,4,4)$ on $K_4\cup K_4$.

$(16,12,9,4,4)\rightarrow$ G-sequence because it contains 9.
\item For 17, we use it up on the edges between $K_7$ and $K_2$, the edges connecting $K_2$ and $K_1$, and the one within $K_2$. Then we use 7 from another term on edges between $K_7$ and $K_1$. Suspicious sequences are:

$(17,16,4,4,4) \rightarrow$ G-sequence as checked in the case with 16.

$(17,11,9,4,4)\rightarrow$ G-sequence as it contains 9.
\item For 21, we use it up on the edges between $K_7$ and $K_2$ and the edges between $K_7$ and $K_1$. We use 3 from another term on the edges connecting $K_2$ and $K_1$, and the one within $K_2$. Suspicious sequences are:

$(21,12,4,4,4)\rightarrow (21,12,4,4,4)-(21,2,1,0,0)=(10,4,4,3)$ on $K_7$

$(21,9,7,4,4)\rightarrow$ G-sequence since it contains 9.
\item For 22, we use it up to color the edges between $K_7$ and $K_2$, the edges between $K_7$ and $K_1$, and the one within $K_2$. We use 2 from another term on the edges connecting $K_2$ and $K_1$. Suspicious sequences are:

$(22,11,4,4,4)\rightarrow(22,11,4,4,4)-(22,0,2,0,0)=(11,4,4,2)$ on $K_7$.

$(22,9,6,4,4)\rightarrow$ G-sequence on $K_{10}$ as it contains 9.
\item For 23, we use it up on the edges between $K_7$ and $K_2$, the edges between $K_7$ and $K_1$, and those connecting $K_2$ and $K_1$. We use 1 from another term on the one within $K_2$. Suspicious sequences are $(23,10,4,4,4)$ and $(23,9,5,4,4)$:

$(23,10,4,4,4)\rightarrow(23,10,4,4,4)-(23,0,1,0,0)=(10,4,4,3)$ on $K_7$.

$(23,9,5,4,4)\rightarrow$ G-sequence on $K_{10}$ as it contains 9.
\item For 24, we use it up on all edges other than those within $K_7$ in our decomposition. The only suspicious sequence is $(24,9,4,4,4)$. We know that it must be a G-sequence as it contains 9.\qedhere
\end{itemize}
\end{proof}

\subsection{Proof of Theorem \ref{g5}}
\begin{proof}
First, $(12,6,6,6,6)$ is not a G-sequence on $K_9$. The only decomposition we can do with sequence is $K_9 \rightarrow K_8 \cup K_1$. What remains afterwards is $(6,6,6,6,4)$ on $K_8$, which violates Corollary \ref{bigbase}. Thus $g(5)>9$.

We now proceed to show that $g(5) = 10$. Excluding values 9,8,7, we divide all of $(10,5)$-sequences into 4 categories.

\textbf{Case 1.} $e_1 \geq 10 > 6\geq e_2$. Then $e_1 \geq 21$ in this case as $e_2+e_3+e_4+e_5 \leq 24$. By Lemmas \ref{big} and \ref{terms}, we're done in this case.

\textbf{Case 2.} $e_1 \geq e_2\geq 10 > 6\geq e_3$. Then we have $e_1 \geq 14$ as $e_3+e_4+e_5 \leq 18$. By Lemma \ref{big} and \ref{terms}, we only have to deal with cases where $e_1=18,19,20$. We use the following the decomposition: $K_{10} \rightarrow K_6\cup K_3 \cup K_1$.
\begin{itemize}
\item When $e_1=18$, we use it up on the edges between $K_6$ and $K_3$. Since $e_2 \geq 10$, we use 9 for edges between $K_1$ and $K_3$, $K_6$ respectively, and 1 for an edge within $K_3$. Then use 2 from a term amongst $e_3,e_4,e_5$ on the remaining two edges within $K_3$. What remains is a $(6,4)$-sequence. Similar to the idea previously, all suspicious sequences in this subcase are:

$(18,13,6,4,4) \rightarrow$ Use $e_1$ as above, $e_2$ for a $K_{1,9}$, and $e_4$ for edges in $K_3\rightarrow (18,13,6,4,4)-(18,9,0,3,0)=(6,4,4,1)$ on $K_6$.

$(18,13,6,5,3)\rightarrow$ Use $e_1$ as above, $e_2$ for a $K_{1,9}$, and $e_4$ for edges in $K_3\rightarrow(18,13,6,5,3)-(18,12,0,0,0)=(6,5,3,1)$ on $K_6$.

$(18,17,5,3,2),(18,17,4,3,3)$,$(18,17,4,4,2),(18,17,6,2,2)$,$(18,14,9,2,2)$,

$(18,14,7,4,2),(18,14,5,4,4),(18,14,6,4,3),(18,13,9,3,2),(18,13,7,5,2)$,

$(18,13,7,4,3),(18,12,9,3,3)(18,12,7,5,3),(18,12,9,4,2),(18,12,7,6,2)$,

$(18,12,7,4,4),(18,13,8,3,3),(18,16,5,3,3)\rightarrow$ All these are G-sequences on $K_{10}$ by Lemma \ref{terms}.
\item When $e_1=19$, we use 18 on the edges between $K_6$ and $K_3$, and the remaining 1 on an edge within $K_3$. We use 9 from $e_2$ for edges between $K_1$ and $K_3$, $K_6$ respectively. Then use 2 from a term amongst $e_3,e_4,e_5$ on the remaining two edges within $K_3$. Suspicious sequences in this subcase are:

$(19,13,6,4,3) \rightarrow$ Use $e_1$ up as above, and $e_2$ for a $K_{1,9}$ and 2 within $K_3\rightarrow (19,13,6,4,3)-(19,11,0,0,0)=(6,4,3,2)$ on $K_6$.

$(19,13,5,4,4)\rightarrow$ Use $e_1$ up as above, and $e_2$ for a $K_{1,9}$ and 2 within $K_3\rightarrow (19,13,5,4,4)-(19,11,0,0,0)=(5,4,4,2)$ on $K_6$.

$(19,12,6,4,4)\rightarrow$ Use $e_1$ up as above, and $e_2$ for a $K_{1,9}$ and 2 within $K_3\rightarrow (19,12,6,4,4)-(19,11,0,0,0)=(6,4,4,1)$ on $K_6$.

$(19,12,6,5,3)\rightarrow$ Use $e_1$ up as above, and $e_2$ for a $K_{1,9}$ and 2 within $K_3\rightarrow (19,12,6,5,3)-(19,11,0,0,0)=(6,5,3,1)$ on $K_6$.

$(19,16,5,3,2),(19,16,4,3,3),(19,16,4,4,2),(19,16,6,2,2),(19,13,9,2,2)$,

$(19,13,7,4,2),(19,12,9,3,2),(19,12,7,5,2),(19,12,7,4,3),(19,11,9,4,2)$,

$(19,11,7,6,2),(19,11,7,4,4),(19,11,9,3,3),(19,11,7,5,3),(19,12,8,3,3)$,

$(19,15,5,3,3)\rightarrow$ G-sequences on $K_{10}$ by Lemma \ref{terms}.
\item When $e_1=20$, we use 18 on the edges between $K_6$ and $K_3$, and the remaining 2 on two edges within $K_3$. We use 10 from $e_2$ for edges between $K_1$ and $K_3$, $K_6$ respectively, and the remaining one edge within $K_3$. Suspicious sequences in this subcase are:

$(20,13,4,4,4)\rightarrow$ Do the decomposition $K_{10} \rightarrow K_5 \cup K_4 \cup K_1 \rightarrow $ Use 20 up on the edges between $K_5$ and $K_4$, and 9 from 13 on the edges between $K_1$ and $K_5$, $K_4$ respectively $\rightarrow (4,4,4,4)$ on $K_5 \cup K_4$.

$(20,13,6,3,3)\rightarrow$ Do the decomposition $K_{10} \rightarrow K_5 \cup K_4 \cup K_1 \rightarrow $ Use 20 up on the edges between $K_5$ and $K_4$, and 9 from 13 on the edges between $K_1$ and $K_5$, $K_4$ respectively $\rightarrow (6,4,3,3)$ on $K_5 \cup K_4$.

$(20,17,3,3,2),(20,17,4,2,2),(20,14,4,4,3),(20,14,7,2,2),(20,13,7,3,2)$,

$(20,12,7,3,3),(20,12,7,4,2),(20,16,3,3,3)\rightarrow$ All are G-sequences on $K_{10}$ by Lemma \ref{terms}.
\end{itemize}

\textbf{Case 3.} $e_1 \geq e_2 \geq e_3 \geq10 > 6\geq e_4$. Then $e_1 \geq 11$ as $e_4+e_5 \leq 12$. By Lemma \ref{terms} and Case 2, we only need to tackle cases where $e_1=11,12,13$.
\begin{itemize}
\item When $e_1=11$, there's only one sequence to check:

$(11,11,11,6,6)$ can be reduced to the $(6,4)$-sequence $(6,4,3,2)$ by removing $K_{1,9},K_{1,8},K_{1,7},K_{1,6}$. Then by Lemma \ref{spec} it's a G-sequence.
\item When $e_1=12$ or 13, we remove a $K_{9}$ from $e_2$ and a $K_{8}$ from $e_3$. What remains is an $(8,5)$-sequence with $e_1=12$ or 13. Then we do the decomposition $K_8 \rightarrow K_6 \cup K_2$.
\item For $e_1=12$, we use it up on edges between $K_6$ and $K_2$, and then use 1 from $e_3$ on the edge within $K_2$. After such coloring, we derive a $(6,4)$-sequence. Similarly, we check all the suspicious sequences (here $12=e_1 \geq e_2\geq e_3$):

$(12,12,12,6,3)\rightarrow$ Use $e_1$ as above, and then use 1 from $e_2$ on the edge within $K_2\rightarrow(12,12,12,6,3)-(12,9+1,8,0,0)=(6,4,3,2)$ on $K_6$.

$(12,12,12,7,2),(12,12,11,7,3),(12,11,11,7,4) \rightarrow$ G-sequences by Lemma \ref{terms}.
\item For $e_1=13$, we use 12 on edges between $K_6$ and $K_2$, and then 1 on the edge within $K_2$. What remains is a $(6,4)$-sequence. Suspicious sequences are (here $13=e_1 \geq e_2\geq e_3$):

$(13,13,12,4,3),(13,13,11,4,4)$. Removing $K_{1,9},K_{1,8},K_{1,7},K_{1,6},K_{1,5}$ following the order $e_3,e_1,e_2,e_2,e_1$ respectively, we get a $(5,3)$-sequence which is always a G-sequence as $g(3)=5$.

$(13,12,12,4,4)\rightarrow (13,12,12,4,4)-(8,9,0,0,0)=(12,5,4,4,3)$ on $K_8\rightarrow$ Then do $K_8 \rightarrow K_6 \cup K_2 \rightarrow$ Use 12 up on the edges between $K_6$ and $K_2$, and 1 from 3 on the edge within the $K_2$ $\rightarrow (5,4,4,2)$ on $K_6$.

$(13,12,11,6,3)\rightarrow(13,12,11,6,3)-(8,9,7,6,0)=(5,4,3,3)$ on $K_6$.

$(13,12,11,7,2),(13,12,10,7,3),(13,11,11,7,3),(13,13,10,7,2)$,\\
$(13,12,11,7,2),(13,11,10,7,4)\rightarrow$ All are G-sequences on $K_{10}$ by Lemma \ref{terms}.
\end{itemize}

\textbf{Case 4.} $e_1 \geq e_2 \geq e_3 \geq e_4\geq10 > 6\geq e_5$.

By Lemma \ref{terms} and Case 3, we only need to examine cases where $e_1=10,11$. In this sense, there only remain five more sequences to check:\newline
$(11,11,11,11,1),(11,11,11,10,2),(11,11,10,10,3),(11,10,10,10,4),(10,10,10,10,5)$. Removing $K_{1,9},K_{1,8},K_{1,7},K_{1,6}$ following the order $e_4,e_3,e_2,e_1$ respectively, we get $(5,4,3,2,1)$ which is obviously a G-sequence on $K_6$.
\end{proof}

\medskip
\noindent{\bf Acknowledgment.} This research was done in the framework of a Budapest Semesters in Mathematics REU program, Summer, 2019. Our advisor was Andr\'as Gy\'arf\'as.

\end{document}